\documentclass[11pt]{amsart}

\usepackage[pdftitle={Bounds on the Index of Self-Shrinking Tori}, unicode]{hyperref}
\usepackage[utf8]{inputenc}
\usepackage{graphicx}

\renewcommand\L{\mathcal L}
\newcommand\n{\mathbf n}
\newcommand\R{\mathbb R}
\let\div\relax
\DeclareMathOperator\div{div}
\DeclareMathOperator\grad{grad}
\DeclareMathOperator\tr{tr}
\newcommand{\abs}[1]{\left\lvert{#1}\right\rvert}

\newcommand\dd[2][]{\frac{d{#1}}{d{#2}}}
\newcommand\ddtwo[2][]{\frac{d^2{#1}}{d{#2}^2}}

\newtheorem{theorem}{Theorem}[section]
\newtheorem{proposition}[theorem]{Proposition}

\theoremstyle{definition}
\newtheorem{definition}[theorem]{Definition}
\newtheorem{notation}[theorem]{Notation}

\author{Yakov Berchenko-Kogan}
\title[Bounds on the Index of Self-Shrinking Tori]{Bounds on the Index of Rotationally Symmetric Self-Shrinking Tori}

\keywords{mean curvature flow, self-shrinkers, Angenent torus}
\subjclass[2020]{53E10}

\begin{document}

\begin{abstract}
  A closed surface evolving under mean curvature flow becomes singular in finite time. Near the singularity, the surface resembles a self-shrinker, a surface that shrinks by dilations under mean curvature flow. If the singularity is modeled on a self-shrinker other than a round sphere or cylinder, then the singularity is unstable under perturbations of the flow. One can quantify this instability using the index of the self-shrinker when viewed as a critical point of the entropy functional.

  In this work, we prove an upper bound on the index of rotationally symmetric self-shrinking tori in terms of their entropy and their maximum and minimum radii. While there have been a few lower bound results in the literature, we believe that this result is the first upper bound on the index of a self-shrinker. Our methods also give lower bounds on the index and the entropy, and our methods give simple formulas for two entropy-decreasing variations whose existence was proved by Liu. Surprisingly, the eigenvalue corresponding to these variations is exactly $-1$. Finally, we present some preliminary results in higher dimensions and six potential directions for future work.
\end{abstract}

\maketitle

\section{Introduction}
Mean curvature flow is a well-studied geometric flow under which a hypersurface $\Sigma\subset\mathbb R^{n+1}$ evolves in such a way as to decrease its area as fast as possible. Under mean curvature flow, each point on the surface moves in the inward normal direction with velocity equal to the mean curvature of the surface at that point. Mean curvature flow has applications to image denoising, and the rich features of mean curvature flow provide a good testing ground for studying geometric flows and nonlinear parabolic partial differential equations more generally.

Some surfaces, such as spheres and cylinders, evolve under mean curvature by dilations. These surfaces, known as \emph{self-shrinkers}, will shrink until they disappear in finite time. Self-shrinkers are particularly important in the study of mean curvature flow because they model the singularities that develop as a surface evolves under mean curvature flow: As the flow approaches a singularity, the surface will locally resemble a self-shrinker. For example, if the initial surface is convex, then the surface will become rounder and rounder as it shrinks to a point, resembling a sphere right before it disappears. Meanwhile, if the initial surface is shaped like a dumbbell, with two large lobes connected by a thin tube, then the thin tube will collapse, splitting the surface into two. Right before the singular time, the surface will locally resemble a cylinder near the singular point.

The world of self-shrinkers is rich and varied. In addition to the classical examples of the sphere, the cylinder, and the plane, there is a self-shrinking torus called the Angenent torus \cite{a92}. More recently, many other examples of self-shrinkers have been found, both compact and noncompact, with or without rotational symmetry, embedded or only immersed, and with low or high genus \cite{d15,dk17,dln18,kkm18,km14,m15,mo11,n09,n10,n14}. In this paper, we focus on rotationally symmetric immersed tori. Conjecturally, the Angenent torus is the only such torus that is embedded \cite{km14}. However, there are infinitely many other examples of rotationally symmetric self-shrinking tori that are only immersed \cite{dk17}.

If we have a mean curvature flow with a singularity modeled on a particular self-shrinker, one can ask if we will get the same kind of singularity if we perturb the initial surface. Colding and Minicozzi \cite{cm12} show that if the self-shrinker is not a round sphere or cylinder, then the singularity is unstable: there is a small variation of the initial surface that changes the type of singularity. For these unstable self-shrinkers, the next question to ask is: How unstable are they? Quantitatively, what is the maximum dimension of a space of variations of the self-shrinker, all of which change the type of singularity that will appear under mean curvature flow?

As we discuss in greater detail in the preliminaries section, there is a Morse-theoretic approach to this question. Huisken \cite{h90} defined a weighted area functional called the \emph{$F$-functional}, and Colding and Minicozzi \cite{cm12} defined a related concept called the \emph{entropy} that accounts for the translational and dilational symmetries of the flow. A surface is a self-shrinker if it is a critical point of the entropy functional, and the \emph{entropy index} or simply \emph{index} of the self-shrinker is the maximum dimension of a space of variations of the self-shrinker that decrease the entropy.

Note that some authors use a different convention for the index. The main point of contention is whether to include the translations and dilations, which decrease the $F$-functional but do not decrease the entropy. In the case of compact surfaces in $\R^3$, these two notions of index simply differ by $4$, accounting for dilation and the three translations. In the formulas in this paper, we do \emph{not} include the translations and dilations in the index; with this convention, the round sphere has index zero.

In addition to Colding and Minicozzi's result \cite{cm12} that round spheres, round cylinders, and planes are the only complete embedded self-shrinkers with index zero, there have been several other index results. We rely heavily on the work of Liu \cite{l16}, who showed that, apart from round spheres, round cylinders, and planes, any rotationally symmetric self-shrinker has index at least three. There are also results that give lower bounds for the index of self-shrinkers in terms of their genus \cite{ai19,irs18,mc15}.

\subsection{Main result} In this paper, we prove upper and lower index bounds for rotationally symmetric tori. To the best of our knowledge, our result is the first known upper bound for the index of an entropy-unstable self-shrinker. We give our strongest bounds in Theorem \ref{thm:finebound}, but, for the sake of simplicity, we present weaker bounds here.

\begin{theorem}\label{thm:coarsebound}
  Let $\Sigma\subset\R^3$ be an immersed rotationally symmetric self-shrinking torus. Let $F(\Sigma)$ denote the entropy of $\Sigma$, let $r_{\min}$ be the smallest distance between a point on $\Sigma$ and the axis of rotation, and let $R$ be the largest distance between a point on $\Sigma$ and the origin.

  Let $i(\Sigma)$ denote the index of $\Sigma$, with the convention that we exclude the translation and dilation variations from the count. Then
  \begin{equation*}
    \frac{3\sqrt{2e}}\pi F(\Sigma)-7<i(\Sigma)<\frac2\pi\frac{F(\Sigma)}{r_{\min}}e^{R^2/4}\left(3+\frac1{r_{\min}}+2R\right)+2R-1.
  \end{equation*}
\end{theorem}

Thus, if we have an upper bound on the entropy of $\Sigma$, a lower bound on its distance to the axis of rotation, and an upper bound on its diameter, we obtain an upper bound on its index.

For the lower bound, we already know from Liu \cite{l16} that $i(\Sigma)\ge3$. The above lower bound improves on Liu's bound when the entropy of $\Sigma$ is sufficiently high, roughly $F(\Sigma)\ge4.5$. The finer bounds in Theorem \ref{thm:finebound} can improve on this even further.

\subsection{Outline of this paper} In Section \ref{sec:preliminaries}, we discuss preliminaries. We begin by discussing mean curvature flow, self-shrinkers, and rescaled mean curvature flow. Next, we summarize some results from \cite{cm12}, defining the \emph{stability operator} $L$ that gives the second variation formula for Huisken's $F$-functional and whose negative eigenvalues we must count in order to determine the index of $\Sigma$. After that, we summarize key results from \cite{l16}, in which Liu applies \cite{cm12} to the rotationally symmetric case. One of the main results we use from \cite{l16} is the decomposition of the stability operator $L$ into its Fourier components $L_k$.

In Section \ref{sec:lk}, we prove the main ingredient of the results in this paper, Theorem \ref{thm:lk}. This theorem gives a simpler formula for $L_k$ in terms of the Laplacian on the torus cross-section with respect to a conformally changed metric. In Section \ref{sec:indexbounds}, we prove our index bounds in Theorem \ref{thm:finebound} by exploiting the fact that we know the eigenvalues of the Laplacian on a one-dimensional manifold explicitly. In Section \ref{sec:entropybounds}, we use our index upper bound techniques against Liu's index lower bound techniques, obtaining entropy upper bounds as a result. Finally, Liu's work proves the existence of three entropy-decreasing variations for rotationally symmetric tori. In Section \ref{sec:eigenfunctions}, we use Theorem \ref{thm:lk} to give simple formulas for two of these variations and to show that they have eigenvalue $-1$. We illustrate these variations in the case of the Angenent torus in Figures \ref{fig:variationQuiver} and \ref{fig:variation3d}.

Lastly, we look to the future. In Section \ref{sec:highdim}, we give some preliminary results in the higher-dimensional setting, and, in Section \ref{sec:future}, we present six potential projects that we hope are promising future directions for this work.

\section{Preliminaries}\label{sec:preliminaries}
\subsection{Hypersurfaces and mean curvature flow}
We discuss the notation and conventions we will use for immersed oriented hypersurfaces $\Sigma\subset\R^{n+1}$.

Let $\n$ denote the unit normal vector to $\Sigma$. Given a point $x\in\Sigma$ and a vector $v\in T_x\R^{n+1}$, let $v^\perp$ denote the scalar projection of $v$ onto $\n$, namely $v^\perp=\langle v,\n\rangle$. Let $v^\top$ denote the projection of $v$ onto $T_x\Sigma$, namely $v^\top=v-v^\perp\n$.

In a coordinate neighborhood of $\Sigma$, let $\{e_1,\dotsc,e_n\}$ denote the coordinate basis. Let $\nabla$ denote the covariant derivative on $\R^{n+1}$, and let $\nabla_i$ denote $\nabla_{e_i}$.

Given a point $x\in\Sigma$, we may choose Riemannian normal coordinates about $x$, in which case the $e_i$ are orthonormal at $x$, and $(\nabla_ie_j)^\top=0$ at $x$.

\begin{definition}
  Let $A_\Sigma$ denote the \emph{second fundamental form} of $\Sigma$. That is, given $v,w\in T_x\Sigma$, let
  \begin{equation*}
    A(v,w)=(\nabla_vw)^\perp
  \end{equation*}
  In coordinates, let
  \begin{equation*}
    a_{ij}=A(e_i,e_j).
  \end{equation*}
\end{definition}

The second fundamental form is symmetric, so $a_{ij}=a_{ji}$. In normal coordinates about $x$, we have at the point $x$ that $\nabla_ie_j=a_{ij}\n$.

\begin{definition}
  Let $H_\Sigma$ denote the \emph{mean curvature} of $\Sigma$, defined with the normalization convention $H_\Sigma=-\tr A_\Sigma$.
\end{definition}

    If we choose coordinates so that the $e_i$ are orthonormal at a particular point $x\in\Sigma$, then, at that point $x$, $H_\Sigma=-\sum_{i=1}^na_{ii}$. In general, $H_\Sigma=-\sum_{i,j=1}^na_{ij}g^{ij}$, where $g^{ij}$ is the inverse of the matrix $g_{ij}=\langle e_i,e_j\rangle$.

\begin{definition}
  A family of surfaces $\Sigma_t$ evolves under \emph{mean curvature flow} if
  \begin{equation*}
    \dot x=-H_\Sigma\n.
  \end{equation*}
  That is, each point on $\Sigma$ moves with speed $H_\Sigma$ in the inward normal direction.
\end{definition}

\subsection{Self-shrinkers}
A surface $\Sigma$ is a self-shrinker if it evolves under mean curvature flow by dilations. For this paper, however, we will restrict this terminology to refer only to surfaces that shrink to the origin in one unit of time. We refer the reader to \cite{cm12,cmp15,h90}.

\begin{definition}
  A surface $\Sigma$ is a \emph{self-shrinker} if $\Sigma_t=\sqrt{-t}\,\Sigma$ is a mean curvature flow for $t<0$.
\end{definition}

A consequence of this definition along with the definition of mean curvature flow is the self-shrinker equation.
\begin{proposition}[\cite{cm12}]
  If $\Sigma$ is a self-shrinker, then
  \begin{equation*}
    H_\Sigma=\tfrac12x^\perp,
  \end{equation*}
  where $x$ is the position vector in $\R^{n+1}$.
\end{proposition}

We also have an extremely useful variational formulation for self-shrinkers in terms of Huisken's $F$-functional.

\begin{definition}
  The \emph{$F$-functional} takes a surface and computes its weighted area via the formula
  \begin{equation*}
    F(\Sigma)=\frac1{(4\pi)^{n/2}}\int_\Sigma e^{-\abs x^2/4}.
  \end{equation*}
\end{definition}

The role of the normalization constant $(4\pi)^{-n/2}$ is to ensure that if $\Sigma$ is a plane through the origin, then $F(\Sigma)=1$.

\begin{definition}
  $\Sigma$ is a \emph{critical point} of $F$ if for any $f\colon\Sigma\to\R$ with compact support, $F$ does not change to first order as we vary $\Sigma$ by $f$ in the normal direction. More precisely, if we let $\Sigma_s=\{x+sf\n\mid x\in\Sigma\}$, then we have $\dd s\bigr\rvert_{s=0}F(\Sigma_s)=0$.
\end{definition}

\begin{proposition}[\cite{cm12}]\label{prop:shrinkercritical}
  $\Sigma$ is a self-shrinker if and only if $\Sigma$ is a critical point of $F$.
\end{proposition}

The definition of $F$ singles out a particular point in space and time. Colding and Minicozzi introduce a related concept called the entropy, which coincides with the $F$-functional on self-shrinkers but is invariant under translations and dilations.
\begin{definition}
  The \emph{entropy} of a hypersurface $\Sigma\subset\R^{n+1}$ is the supremum of the $F$-functional evaluated on all translates and dilates of $\Sigma$, that is $\sup_{x_0,t_0}F(x_0+\sqrt{t_0}\Sigma)$.
\end{definition}
If $\Sigma$ is a self-shrinker, defined as above to shrink to the origin in one unit of time, then the supremum among translates and dilates is attained at $\Sigma$ itself, so the entropy of $\Sigma$ coincides with $F(\Sigma)$. However, entropy-decreasing variations of $\Sigma$ and $F$-decreasing variations of $\Sigma$ are not quite the same: when we ask about entropy-decreasing variations, we exclude the ``trivial'' $F$-decreasing variations of translation and dilation.

\subsection{Rescaled mean curvature flow}
The gradient flow for the $F$-functional is rescaled mean curvature flow. While rescaled mean curvature flow is not necessary for our results, we feel it provides valuable context, so we briefly introduce it here.

\begin{definition}
  A family of surfaces $\tilde\Sigma_\tau$ evolves by rescaled mean curvature flow if $\tilde\Sigma_\tau=\frac1{\sqrt{-t}}\Sigma_t$, where $t=-e^{-\tau}$ and $\Sigma_t$ evolves by mean curvature flow.
\end{definition}

The definition of rescaled mean curvature flow is designed so that as $\Sigma_t$ approaches a singularity at the origin at time $t=0$, we zoom in on the origin and reparametrize time so that the rescaled surface $\tilde\Sigma_\tau$ remains a fixed size as $\tau\to\infty$. In particular, it is easy to see a self-shrinker is a stationary point for rescaled mean curvature flow.

Rescaled mean curvature flow has the following evolution equation.
\begin{proposition}
  A family of surfaces $\tilde\Sigma_\tau$ evolves by rescaled mean curvature flow if the points $x\in\tilde\Sigma_\tau$ flow according to
  \begin{equation*}
    \dd\tau x=\left(-H_{\tilde\Sigma}+\tfrac12x^\perp\right)\n.
  \end{equation*}
\end{proposition}

The fact that self-shrinkers are critical points of the $F$-functional is a special case of the fact that rescaled mean curvature flow is the gradient flow of $F$, in the following sense.
\begin{proposition}[\cite{cm12}]
  Let $\Sigma$ be a surface, let $f\colon\Sigma\to\R$, and let $\Sigma_s$ be the normal variation of $\Sigma$ corresponding to $f$, namely $\Sigma_s=\{x+sf\n\mid x\in\Sigma\}$. Then
  \begin{equation*}
    \dd s\Bigr\rvert_{s=0}F(\Sigma_s)=-\frac1{(4\pi)^{n/2}}\int_\Sigma\left(-H_\Sigma+\tfrac12x^\perp\right)f\,e^{-\abs x^2/4}.
  \end{equation*}
\end{proposition}

\subsection{The stability operator}

Given a critical point of a flow, the next natural question to ask is about the stability of that critical point. If we perturb a self-shrinker, will the resulting surface flow back to the self-shrinker under rescaled mean curvature flow, or will it flow to a different critical point? What is the maximum dimension of a space of unstable variations? For a gradient flow, answering this question amounts to computing the eigenvalues of the Hessian of the function $F$. Colding and Minicozzi compute this second derivative.

\begin{definition}
  Let the \emph{drift Laplacian} or \emph{Bakry--Emery Laplacian} or \emph{Witten Laplacian} $\L_\Sigma$ be defined by
  \begin{equation*}
    \L_\Sigma f=e^{\abs x^2/4}\div_\Sigma\left(e^{-\abs x^2/4}\grad_\Sigma f\right).
  \end{equation*}
  for $f\colon\Sigma\to\R$.
\end{definition}
The motivation for this definition is that $\L_\Sigma$ has an integration by parts formula analogous to that of the Laplacian. Namely, for compactly supported $f$ and $g$, we have
\begin{equation*}
  \int_\Sigma g(-\L_\Sigma)f\,e^{-\abs x^2/4}=\int_\Sigma\left\langle\grad_\Sigma g,\grad_\Sigma f\right\rangle\,e^{-\abs x^2/4}.
\end{equation*}

\begin{definition}
  Let the \emph{stability operator} $L_\Sigma$ acting on functions $f\colon\Sigma\to\R$ be defined by
  \begin{equation*}
    L_\Sigma=\L_\Sigma+\abs{A_\Sigma}^2+\tfrac12.
  \end{equation*}
\end{definition}

\begin{proposition}[\cite{cm12}]
  Let $\Sigma$ be a self-shrinker, let $f\colon\Sigma\to\R$, and let $\Sigma_s$ be the normal variation of $\Sigma$ corresponding to $f$, namely $\Sigma_s=\{x+sf\n\mid x\in\Sigma\}$. Then
  \begin{equation*}
    \ddtwo s\Bigr\rvert_{s=0}F(\Sigma_s)=\frac1{(4\pi)^{n/2}}\int_\Sigma f(-L_\Sigma)f\,e^{-\abs x^2/4}.
  \end{equation*}
\end{proposition}

\begin{definition}[Sign convention for eigenvalues]
  We use the sign convention $\Delta=\div\grad=-d^*d$, and consequently we will say that $f\neq0$ is an eigenfunction of a differential operator $L$ with eigenvalue $\lambda$ if $-Lf = \lambda f$.
\end{definition}

We conclude that eigenfunctions of the stability operator $L_\Sigma$ with negative eigenvalues are unstable variations of the self-shrinker $\Sigma$: If we vary $\Sigma$ in that direction, then rescaled mean curvature flow will take the surface away from $\Sigma$. Meanwhile, eigenfunctions of the stability operator $L_\Sigma$ with positive eigenvalues are stable variations of $\Sigma$: There exists a rescaled mean curvature flow line that approaches $\Sigma$ from that direction.

The variations corresponding to translating or dilating a self-shrinker $\Sigma$ are unstable eigenfunctions of $L_\Sigma$, due to the following geometric reason: Translating $\Sigma$ moves the location of the singularity in space, and dilating $\Sigma$ moves the location of the singularity in time. Since rescaled mean curvature flow zooms in on the origin as time approaches $t=0$, one can check that under rescaled mean curvature flow the translated or dilated surface will continue translating or dilating away from $\Sigma$, respectively.

Translating $\Sigma$ in the direction $v\in R^{n+1}$ corresponds to the normal variation $f=v^\perp$, and dilating $\Sigma$ corresponds to the normal variation $f=H_\Sigma$. Colding and Minicozzi compute the corresponding eigenvalues of the stability operator.
\begin{proposition}[\cite{cm12}]
  For any vector $v\in\R^{n+1}$, we have $L_\Sigma v^\perp=\frac12v^\perp$. Meanwhile, for dilation, we have $L_\Sigma H_\Sigma=H_\Sigma$.
\end{proposition}
Thus, assuming these functions are nonzero, $v^\perp$ and $H_\Sigma$ are eigenfunctions of $L_\Sigma$, giving us $n+2$ independent eigenfunctions. With our sign convention, the eigenvalue corresponding to $v^\perp$ is $-\frac12$, and the eigenvalue corresponding to $H_\Sigma$ is $-1$.

Because $L_\Sigma$ has the same symbol as $\Delta_\Sigma$, it has a finite number of negative eigenvalues, at least in the case of compact $\Sigma$. Usually, one defines the index of a critical point of a gradient flow to be the number of negative eigenvalues of the Hessian. However, because translations and dilations do not change the shape of the self-shrinker, we exclude them in this context.

\begin{definition}
  The \emph{index} of a self-shrinker $\Sigma$ is the number of negative eigenvalues of the stability operator $L_\Sigma$, excluding those eigenvalues corresponding to translations and dilations.
\end{definition}

Assuming that $\Sigma$ is not invariant under any translations or dilations, its index is simply $n+2$ less than the usual Morse index.

Under mild assumptions, Colding and Minicozzi show that the only self-shrinkers with index zero are planes, round spheres, and round cylinders \cite{cm12}.

\subsection{Rotationally symmetric self-shrinkers}
If $\Sigma\subset\R^{n+1}$ is a hypersurface with $SO(2)$ rotational symmetry, we can understand it in terms of its cross-section $\Gamma$. We refer the reader to \cite{l16}.

We will use cylindrical coordinates on $\R^{n+1}=\R^2\times\R^{n-1}$ in the sense that $x=(r,\theta,z)$, where $(r,\theta)$ describe the first two components of $x$ in polar coordinates, and $z\in\R^{n-1}$ represents the remaining $n-1$ components of $x$. We let $e_r$ and $e_\theta$ denote the unit vectors in the radial and angular directions, respectively.

\begin{definition}
  We say that a hypersurface $\Sigma\subset\R^{n+1}=\R^2\times\R^{n-1}$ is \emph{$SO(2)$-rotationally symmetric} or simply \emph{rotationally symmetric} if it is invariant under the action of $SO(2)$ on the first two coordinates of $\R^{n+1}$.
\end{definition}

Note that for $n>2$, the rotational symmetry we consider is different from the $SO(n)$ rotational symmetry discussed in papers such as \cite{a92}.

If $\Sigma$ is rotationally symmetric, we let $\Gamma$ denote its $\theta=0$ cross-section, which we also think of as being an $(n-1)$-dimensional hypersurface in the halfspace $\{(r,z)\mid r\ge0,z\in\R^{n-1}\}$. Choosing normal coordinates in $\Gamma$ about $x\in\Gamma$, we obtain a frame $e_1,\dotsc,e_{n-1}$ for $\Gamma$. By rotating this frame along with $\Gamma$ and appending $e_\theta$, we obtain a local frame for $\Sigma$ that is orthonormal at $x$. Since $e_\theta$ is constant on $\Gamma$, we have that $\nabla_ie_\theta=0$ for $1\le i\le n-1$, and we can compute that $\nabla_\theta e_\theta=-\frac1r e_r$. Consequently, we can relate the second fundamental forms and mean curvatures of $\Sigma$ and $\Gamma$.

\begin{proposition}
  We have $a_{i\theta}=0$ for $i=1,\dotsc,n-1$, and $a_{\theta\theta}=-\frac1re_r^\perp$. Consequently,
  \begin{align*}
    H_\Sigma&=H_\Gamma+\frac1re_r^\perp,\\
    \abs{A_\Sigma}^2&=\abs{A_\Gamma}^2+\frac1{r^2}\lvert e_r^\perp\rvert^2.
  \end{align*}
\end{proposition}
We can thus rewrite the self-shrinker equation in terms of the cross-section $\Gamma$.
\begin{proposition}
  $\Gamma$ is the cross-section of a self-shrinker $\Sigma$ if
  \begin{equation*}
    H_\Gamma=\frac12x^\perp-\frac1re_r^\perp.
  \end{equation*}
\end{proposition}

We can also write the $F$-functional in terms of $\Gamma$.
\begin{proposition}\label{prop:FGamma}
  If $\Sigma$ is a rotationally symmetric hypersurface with cross-section $\Gamma$, then
  \begin{equation*}
    F(\Sigma)=\frac{2\pi}{(4\pi)^{n/2}}\int_\Gamma re^{-\abs x^2/4}.
  \end{equation*}
\end{proposition}

To simplify our notation, we will let $\sigma$ denote this weight.
\begin{definition}
  Let $\sigma\colon\R_{\ge0}\times\R^{n-1}\to\R$ denote the weight
  \begin{equation*}
    \sigma=\frac{2\pi}{(4\pi)^{n/2}}re^{-\abs x^2/4}.
  \end{equation*}
\end{definition}

\subsection{The stability operator for rotationally symmetric self-shrinkers}
Varying the cross-section $\Gamma$ only yields rotationally symmetric variations of $\Sigma$. To understand the stability operator $L_\Sigma$ in terms of $\Gamma$, we must understand non-rotationally symmetric variations of $\Sigma$ as well. Liu \cite{l16} does so by decomposing normal variations $f\colon\Sigma\to\R$ into their Fourier components. The stability operator $L_\Sigma$ commutes with this Fourier decomposition, so we can decompose $L_\Sigma$ into its Fourier components $L_k$, which are operators acting on functions on $\Gamma$. We summarize these key results from \cite{l16} in this subsection.

We begin by defining the drift Laplacian on $\Gamma$ using the weight $\sigma=\frac{2\pi}{(4\pi)^{n/2}}re^{-\abs x^2/4}$.
\begin{definition}
  Let the \emph{drift Laplacian} $\L_\Gamma$ be defined by
  \begin{equation*}
    \L_\Gamma u=\sigma^{-1}\div_\Gamma\left(\sigma\grad_\Gamma u\right)
  \end{equation*}
  for $u\colon\Gamma\to\R$.
\end{definition}
A quick computation yields the explicit form
\begin{equation*}
  \L_\Gamma u=\Delta_\Gamma u + \left\langle\tfrac1re_r-\tfrac12x,\grad_\Gamma u\right\rangle.
\end{equation*}
\begin{definition}\label{def:lk}
  Let the $k$th Fourier component of the stability operator be
  \begin{equation*}
    L_k=\L_\Gamma+\abs{A_\Gamma}^2+\frac1{r^2}\lvert e_r^\perp\rvert^2+\frac12-\frac{k^2}{r^2},
  \end{equation*}
  acting on functions $u\colon\Gamma\to\R$.
\end{definition}

\begin{proposition}[\cite{l16}]\label{prop:Lfourierdecomposition}
  Let $u\colon\Gamma\to\R$, let $k$ be a nonnegative integer, and let $f\colon\Sigma\to\R$ be defined by $f=u\cos k\theta$, in the sense that $f(r,\theta,z)=u(r,z)\cos k\theta$. Then
  \begin{equation*}
    L_\Sigma f=(L_ku)\cos k\theta,
  \end{equation*}
  and likewise for sine in place of cosine.
\end{proposition}

Thus, we can determine the eigenvalues and eigenfunctions of the stability operator $L_\Sigma$ by determining the eigenvalues and eigenfunctions of the operators $L_k$ for all $k\ge0$.

From here, Liu \cite{l16} obtains his index results as follows: For each $k$, the eigenfunction corresponding to the least eigenvalue of $L_k$ cannot change sign. Dilation is given by the variation $H_\Sigma$, which is rotationally symmetric, so $H_\Sigma$ is an eigenfunction of $L_0$ with eigenvalue $-1$. Thus, if $H_\Sigma$ changes sign, then there must exist another eigenfunction of $L_0$ with smaller eigenvalue. Likewise, horizontal translation is given by the variations $e_r^\perp\cos\theta$ and $e_r^\perp\sin\theta$ in the first Fourier component, so $e_r^\perp$ is an eigenfunction of $L_1$ with eigenvalue $-\frac12$. Thus, if $e_r^\perp$ changes sign, there must exist another eigenfunction $u$ of $L_1$ with smaller eigenvalue, giving a pair of eigenfunctions $u\cos\theta$ and $u\sin\theta$ of $L$. Liu concludes, under assumptions sufficient to guarantee that $H_\Sigma$ and $e_r^\perp$ change sign, that the index of $\Sigma$ is at least three.

\section{A simple formula for $L_k$ in terms of \texorpdfstring{$\Delta_\Gamma^\sigma$}{∆}}\label{sec:lk}
In this section, we will restrict our attention to immersed rotationally symmetric tori, so $n=2$, the cross-section $\Gamma$ is an immersed closed curve that stays away from the axis of rotation $r=0$, and $\sigma=\frac12re^{-\abs x^2/4}$.

In Definition \ref{def:lk}, we presented Liu's formula for the $k$th Fourier component $L_k$ of the stability operator. In this section, we give a simpler formula for $L_k$ in terms of the Laplacian on the cross-section $\Gamma$ with respect to the metric $g^\sigma=\sigma^2(dr^2+dz^2)$. In addition to its simplicity, this formula is advantageous because we know the eigenvalues of the Laplacian on a closed curve explicitly: they depend only on the length of the curve, which in this case is the entropy $F(\Sigma)$ by Proposition \ref{prop:FGamma}. This eigenvalue information is the key to proving our index bounds in Section \ref{sec:indexbounds} and our entropy bounds in Section \ref{sec:entropybounds}.

\begin{notation}
  By default, we will work with respect to the standard Euclidean metric. However, whenever we place a superscript $\sigma$, we are working with respect to the metric $g^\sigma$ on $\R_{>0}\times\R$. For instance, $\Delta_\Gamma$ denotes the Laplacian on $\Gamma$ with respect to the Euclidean metric, whereas $\Delta_\Gamma^\sigma$ denotes the Laplacian on $\Gamma$ with respect to the metric $g^\sigma$.
\end{notation}

A variation of $\Gamma$ corresponds to a rotationally symmetric variation of $\Sigma$. In particular, if $\Sigma$ is a self-shrinker, then $\Sigma$ is a critical point for the functional $F$, so the curve $\Gamma$ is a critical point for the length functional $l^\sigma$ with respect to the metric $g^\sigma$. In other words, $\Gamma$ is a geodesic in the half-plane with the metric $g^\sigma$. As such, we can assess the stability of $\Gamma$ using the Jacobi operator. We will first need a couple basic quantities with respect to the metric $g^\sigma$.

\begin{definition}
  Let $N$ denote the unit normal vector to $\Gamma$ with respect to the metric $g^\sigma$.
\end{definition}
One can check that $N=\sigma^{-1}\n$, where we recall from Section \ref{sec:preliminaries} that $\n$ is the unit normal with respect to the Euclidean metric.
\begin{definition}
  Let $K$ denote the Gaussian curvature of $\R_{>0}\times\R$ with respect to the metric $\sigma$.
\end{definition}
A computation gives
\begin{equation*}
  K=-\sigma^{-2}\Delta\log\sigma=\sigma^{-2}\left(1+\frac1{r^2}\right).
\end{equation*}
\begin{definition}
  We suggestively denote by $L_0^\sigma$ the Jacobi operator defined by the formula
  \begin{equation*}
    L_0^\sigma=\Delta_\Gamma^\sigma+K=\Delta_\Gamma^\sigma+\sigma^{-2}\left(1+\frac1{r^2}\right).
  \end{equation*}
\end{definition}
Then, either from the theory of Jacobi fields or as a special case of the second variation formula for minimal hypersurfaces, we see that $L_0^\sigma$ gives the Hessian of the length functional $l^\sigma$.
\begin{proposition}
  Let $\Gamma$ be a geodesic in $\R_{>0}\times\R$ with respect to the metric $g^\sigma$. Let $\eta\colon\Gamma\to\R$, and let $\Gamma_s$ be the normal variation of $\Gamma$ corresponding to $\eta$, namely $\Gamma_s=\{x+s\eta N\mid x\in\Gamma\}$. Then
  \begin{equation*}
    \ddtwo s\Bigr\rvert_{s=0} l^\sigma(\Gamma_s)=\int_\Gamma\eta(-L_0^\sigma)\eta\,\sigma.
  \end{equation*}
\end{proposition}
As suggested by the notation, we have a simple relationship between $L_0$ and $L_0^\sigma$.
\begin{proposition}\label{prop:l0sigma}
  The stability operators $L_0$ and $L_0^\sigma$ are related by the formula
  \begin{equation*}
    L_0=\sigma L_0^\sigma\sigma=\sigma\Delta_\Gamma^\sigma\sigma+1+\frac1{r^2}.
  \end{equation*}
\end{proposition}
\begin{proof}
    Rather than comparing this formula for $L_0$ with the formula for $L_0$ in Definition \ref{def:lk}, we will instead observe that $L_0$ and $L_0^\sigma$ are representing the same symmetric bilinear operator, namely the Hessian of $l^\sigma=F$ with respect to variations of $\Gamma$. This proposition is essentially a change of basis formula transforming between writing normal variations as $u\n$ and writing normal variations as $\eta N$. 

    More explicitly, we set $u=\sigma^{-1}\eta\colon\Gamma\to\R$. With this identification, $u\n=\eta N$. Letting $f\colon\Sigma\to\R$ be the rotationally symmetric variation $f=u$, we see that the variations
    \begin{equation*}
      \Sigma_s=\{x+sf\n\mid x\in\Sigma\}
    \end{equation*}
    have cross-sections
    \begin{equation*}
      \Gamma_s=\{x+s\eta N\mid x\in\Gamma\}.
    \end{equation*}
    Consequently,
    \begin{multline*}
      \int_\Gamma u(-L_0)u\,\sigma=\frac1{4\pi}\int_\Sigma f(-L_\Sigma)f\,e^{-\abs x^2/4}=\ddtwo s\Bigr\rvert_{s=0}F(\Sigma_s)\\=\ddtwo s\Bigr\rvert_{s=0} l^\sigma(\Gamma_s)=\int_\Gamma\eta(-L_0^\sigma)\eta\,\sigma=\int_\Gamma u(-\sigma L_0^\sigma\sigma)u\,\sigma.
    \end{multline*}
    The Hessian is symmetric, so the above equality is sufficient to conclude that $L_0=\sigma L_0^\sigma\sigma$, as desired.
\end{proof}
We immediately obtain a nice formula for $L_k$.
\begin{theorem}\label{thm:lk}
  Let $\Sigma\subset\R^3$ be an immersed rotationally symmetric self-shrinking torus, and let $\Gamma$ be its cross-section. Let $L_k$ be the $k$th Fourier component of stability operator, as given in Definition \ref{def:lk} and Proposition \ref{prop:Lfourierdecomposition} following \cite{l16}. Let $\Delta_\Gamma^\sigma$ denote the Laplacian on $\Gamma$ with respect to the conformally changed metric $g^\sigma=\sigma^2(dr^2+dz^2)$ with $\sigma=\frac 12re^{-\abs x^2/4}$.

  These two operators are related by the formula
  \begin{equation*}
    L_k=\sigma\Delta_\Gamma^\sigma\sigma+1+\frac{1-k^2}{r^2}.
  \end{equation*}
\end{theorem}
\begin{proof}
    From the formula given for $L_k$ in Definition \ref{def:lk}, we have $L_k=L_0-\frac{k^2}{r^2}$. The result then follows from Proposition \ref{prop:l0sigma}.
\end{proof}
  
Proposition \ref{prop:l0sigma} and Theorem \ref{thm:lk} naturally give rise to the definition of $L_k^\sigma$.
\begin{definition}\label{def:lksigma}
  For nonnegative integer $k$, let
  \begin{equation*}
    L_k^\sigma=\sigma^{-1}L_k\sigma^{-1}=\Delta_\Gamma^\sigma+\sigma^{-2}\left(1+\frac{1-k^2}{r^2}\right).
  \end{equation*}
\end{definition}

\section{Index bounds}\label{sec:indexbounds}
As discussed in the preliminaries, to find the index of $\Sigma$, we must find the number of negative eigenvalues of each $L_k$ for $k\ge0$. Because $\sigma>0$, the number of negative eigenvalues of $L_k$ is the same as the number of negative eigenvalues of $L_k^\sigma=\sigma^{-1}L_k\sigma^{-1}$, but the operator $L_k^\sigma$ is easier to analyze.

This section is structured as follows. In Subsection \ref{subsec:delta}, we use the spectrum of $\Delta_\Gamma^\sigma$ to give us information about the spectrum of $L_k^\sigma$. This information gives us bounds on the number of negative eigenvalues of $L_k^\sigma$, or, equivalently, of $L_k$, in terms of the cross-section $\Gamma$. We write down these bounds, as well as the resulting bounds on the index of $\Sigma$, in Subsection \ref{subsec:fine}. In particular, we give our stronger index bounds in Theorem \ref{thm:finebound}. Then, in Subsection \ref{subsec:coarse}, we write down weaker bounds that depend only on the entropy of $\Sigma$, its minimum distance to the axis of rotation, and its maximum distance to the origin, proving Theorem \ref{thm:coarsebound}. Finally, in \ref{subsec:indexbounds}, we illustrate the results of Subsection \ref{subsec:fine} by applying them to the Angenent torus.

\subsection{Negative eigenvalues of $L_k^\sigma$}\label{subsec:delta}
The first term of $L_k^\sigma$ is $\Delta_\Gamma^\sigma$, which is the Laplacian of a one-dimensional manifold, so its spectrum is a classical result.
\begin{proposition}\label{prop:deltaeigs}
  Let $l=l^\sigma(\Gamma)=F(\Sigma)$. Let $s\colon[0,l]\to\Gamma$ be a parametrization of $\Gamma$ with respect to $\sigma$-arclength. Then the eigenfunctions of $\Delta_\Gamma^\sigma$ are given by
  \begin{align*}
    -\Delta_\Gamma^\sigma1&=0,&-\Delta_\Gamma^\sigma\cos\left(j\tfrac{2\pi}ls\right)&=\left(j\tfrac{2\pi}l\right)^2\cos\left(\tfrac{2\pi}ls\right),\\ &&-\Delta_\Gamma^\sigma\sin\left(j\tfrac{2\pi}ls\right)&=\left(j\tfrac{2\pi}l\right)^2\sin\left(\tfrac{2\pi}ls\right).
  \end{align*}
  Thus, the spectrum of $\Delta_\Gamma^\sigma$ is
  \begin{align*}
    \lambda_0&=0,&\lambda_{2j-1}=\lambda_{2j}&=\left(\tfrac{2\pi}l\right)^2j^2,\qquad j\ge1.
  \end{align*}
\end{proposition}
The second term of $L_k^\sigma$ is $\sigma^{-1}\left(1+\frac{1-k^2}{r^2}\right)$, which is a bounded zeroth order operator. Thus, we can obtain upper and lower bounds on the number of negative eigenvalues of $L_k^\sigma$.
\begin{proposition}\label{prop:maxeigs}
  Let $J$ be the smallest nonnegative integer such that
  \begin{equation*}
    J^2\ge\left(\frac l{2\pi}\right)^2\max_\Gamma\sigma^{-2}\left(1+\frac{1-k^2}{r^2}\right).
  \end{equation*}
  If $J=0$, then $L_k^\sigma$ has no negative eigenvalues. If $J\ge1$, then $L_k^\sigma$ has at most $2J-1$ negative eigenvalues.
\end{proposition}
\begin{proof}
    If $J=0$, let $\eta$ be arbitrary. If $J\ge1$, let $\eta$ be orthogonal to the first $2J-1$ eigenfunctions of $\Delta_\Gamma^\sigma$ with respect to the $L^2(\Gamma,\sigma)$ pairing. On this orthogonal complement, the smallest eigenvalue of $\Delta_\Gamma^\sigma$ is $\left(\frac{2\pi}l\right)J^2$. Thus,
    \begin{multline*}
      \int_\Gamma\eta(-L_k^\sigma)\eta\,\sigma=\int_\Gamma\left(\eta(-\Delta_\Gamma^\sigma)\eta-\sigma^{-2}\left(1+\frac{1-k^2}{r^2}\right)\eta^2\right)\,\sigma\\\ge\int_\Gamma\left(\left(\frac{2\pi}l\right)^2J^2-\sigma^{-2}\left(1+\frac{1-k^2}{r^2}\right)\right)\eta^2\,\sigma\ge0.
    \end{multline*}
    Thus, $L_k^\sigma$ has at most $2J-1$ negative eigenfunctions if $J\ge1$, and $L_k^\sigma$ has no negative eigenfunctions if $J=0$.
\end{proof}

\begin{proposition}\label{prop:mineigs}
  Let $k^2\le1+r_{\min}^2$, so $1+\frac{1-k^2}{r^2}\ge0$ for all points on $\Gamma$. Let $J$ be the largest integer such that
  \begin{equation*}
    J^2\le\left(\frac l{2\pi}\right)^2\min_\Gamma\sigma^{-2}\left(1+\frac{1-k^2}{r^2}\right)
  \end{equation*}
  and such that the inequality $J^2<\left(\frac l{2\pi}\right)^2\sigma^{-2}\left(1+\frac{1-k^2}{r^2}\right)$ is strict somewhere on $\Gamma$.
  
  Then $L_k^\sigma$ has at least $2J+1$ negative eigenvalues.
\end{proposition}
\begin{proof}
    Let $\eta$ be a nonzero linear combination of the $2J+1$ eigenfunctions of $\Delta^\sigma_\Gamma$ with eigenvalue at most $\left(\frac{2\pi}l\right)J^2$. Then
    \begin{multline*}
      \int_\Gamma\eta(-L_k^\sigma)\eta\,\sigma=\int_\Gamma\left(\eta(-\Delta_\Gamma^\sigma)\eta-\sigma^{-2}\left(1+\frac{1-k^2}{r^2}\right)\eta^2\right)\,\sigma\\\le\int_\Gamma\left(\left(\frac{2\pi}l\right)^2J^2-\sigma^{-2}\left(1+\frac{1-k^2}{r^2}\right)\right)\eta^2\,\sigma<0.
    \end{multline*}
    Thus, $L_k^\sigma$ has at least $2J+1$ negative eigenvalues.
\end{proof}

\subsection{Fine index bounds}\label{subsec:fine}
We can re-express the results of Propositions \ref{prop:maxeigs} and \ref{prop:mineigs} as follows.
\begin{proposition}\label{prop:fineboundk}
  Let $i_k$ denote the number of negative eigenvalues of $L_k$, or, equivalently, of $L_k^\sigma$. Then
  \begin{align}
    i_k&\ge2\left\lfloor\frac l{2\pi}\min_\Gamma\sigma^{-1}\sqrt{1+\frac{1-k^2}{r^2}}\right\rfloor+1,&k^2&\le1+r_{\min}^2,\label{eq:lowerbound}\\
    i_k&\le2\left\lceil\frac l{2\pi}\max_\Gamma\sigma^{-1}\sqrt{1+\frac{1-k^2}{r^2}}\right\rceil-1,&k^2&<1+r_{\max}^2,\label{eq:upperbound}\\
    i_k&=0,&k^2&\ge1+r_{\max}^2,\label{eq:noeigs}
  \end{align}
  except in the case where $\frac l{2\pi}\sigma^{-1}\sqrt{1+\frac{1-k^2}{r^2}}$ is constant on $\Gamma$ and an integer, in which case \eqref{eq:lowerbound} becomes $i_k\ge2\left(\frac l{2\pi}\sigma^{-1}\sqrt{1+\frac{1-k^2}{r^2}}\right)-1$.

  As always, the notation $\lfloor a\rfloor$ and $\lceil a\rceil$ denotes the largest integer less than or equal to $a$ and the smallest integer greater than or equal to $a$, respectively. In \eqref{eq:upperbound}, by $\max_\Gamma$, we mean the maximum value over all points of $\Gamma$ where the expression is defined, which is a nonempty region by the assumption that $k^2<1+r_{\max}^2$.
\end{proposition}
\begin{proof}
    Equation \eqref{eq:lowerbound} follows directly from Proposition \ref{prop:mineigs} by taking the square root. When $\frac l{2\pi}\sigma^{-1}\sqrt{1+\frac{1-k^2}{r^2}}$ is a constant integer, we must lower $J$ by one to ensure that the inequality $J^2<\left(\frac l{2\pi}\right)^2\sigma^{-2}\left(1+\frac{1-k^2}{r^2}\right)$ is strict somewhere on $\Gamma$.

    Next, looking at Proposition \ref{prop:maxeigs}, we note that if $k^2\ge1+r_{\max}^2$, then $1+\frac{1-k^2}{r^2}\le0$ on all of $\Gamma$, so $J=0$, giving us \eqref{eq:noeigs}. On the other hand, if $k^2<1+r_{\max}^2$, then $1+\frac{1-k^2}{r_{\max}^2}>0$, so $J>0$, and so \eqref{eq:upperbound} follows by taking the square root of the inequality defining $J$.
\end{proof}

We thus obtain our index bounds.
\begin{theorem}\label{thm:finebound}
  Let $\Sigma\subset\R^3$ be an immersed rotationally symmetric self-shrinking torus. Let $F(\Sigma)$ denote the entropy of $\Sigma$, let $r$ be the distance from the axis of rotation, and let $\sigma=\frac12re^{-\abs x^2/4}$. Let $i(\Sigma)$ be the index of $\Sigma$, with the convention that we exclude the translation and dilation variations from the count. Then
  \begin{align}
    \begin{split}
      i(\Sigma)&\le2\left\lceil\frac{F(\Sigma)}{2\pi}\max_\Sigma\sigma^{-1}\sqrt{1+\frac{1}{r^2}}\right\rceil-5\\
      &\phantom{\le}+\sum_{1\le k<\sqrt{1+r_{\max}^2}}\left(4\left\lceil\frac{F(\Sigma)}{2\pi}\max_\Sigma\sigma^{-1}\sqrt{1+\frac{1-k^2}{r^2}}\right\rceil-2\right),
    \end{split}\\
    \begin{split}\label{eq:fulllower}
      i(\Sigma)&\ge2\left\lfloor\frac {F(\Sigma)}{2\pi}\min_\Sigma\sigma^{-1}\sqrt{1+\frac{1}{r^2}}\right\rfloor-3\\
      &\phantom{\le}+\sum_{1\le k\le\sqrt{1+r_{\min}^2}}\left(4\left\lfloor\frac {F(\Sigma)}{2\pi}\min_\Sigma\sigma^{-1}\sqrt{1+\frac{1-k^2}{r^2}}\right\rfloor+2\right),
    \end{split}                 
  \end{align}
  except if for some value of $k$, $\frac{F(\Sigma)}{2\pi}\sigma^{-1}\sqrt{1+\frac{1-k^2}{r^2}}$ is constant on $\Sigma$ and an integer. If this value of $k$ is zero, we must replace $-3$ with $-5$ in the lower bound \eqref{eq:fulllower}. If this value of $k$ is positive, we must replace $-3$ with $-7$.
\end{theorem}
\begin{proof}
  A compact self-shrinker is not invariant under translations and dilations, so, per our definition, the index of $\Sigma$ is four less than the number of negative eigenvalues of $L$. By Proposition \ref{prop:Lfourierdecomposition}, each negative eigenvalue of $L_0$ gives a negative eigenvalue of $L$, and each negative eigenvalue of $L_k$ gives a pair of negative eigenvalues of $L_0$, one for $u\cos k\theta$ and one for $u\sin k\theta$. Thus, $i(\Sigma)=-4+i_0+2\sum_{k\ge1}i_k$, and the result follows from Proposition \ref{prop:fineboundk}. With regards to the fine print, we note that for a self-shrinking torus, $r$ is not constant, so the exceptional case where $\frac l{2\pi}\sigma^{-1}\sqrt{1+\frac{1-k^2}{r^2}}$ is constant and an integer can happen for at most one value of $k$.
\end{proof}

\subsection{Coarse index bounds}\label{subsec:coarse}
We now prove coarser but simpler bounds on the index that depend only on the entropy of $\Sigma$, its minimum distance $r_{\min}$ from the axis of rotation, and its maximum distance $R$ from the origin.

\begin{proposition}\label{prop:coarseboundk}
  Let $R=\max_\Gamma\abs x$ and $l$ be the $\sigma$-length of $\Gamma$. Then the number $i_k$ of negative eigenvalues of $L_k$ satisfies the bounds
  \begin{align*}
    \frac l\pi\sqrt{2e}\sqrt{1+\frac1{R^2}}-1&<i_0<\frac l\pi\frac2{r_{\min}}e^{R^2/4}\sqrt{1+\frac1{r_{\min}^2}}+1,\\
    \frac l\pi\sqrt{2e}-1&<i_1<\frac l\pi\frac2{r_{\min}}e^{R^2/4}+1,\\    
    &\phantom{{}<{}}i_k<\frac l\pi\frac2{r_{\min}}e^{R^2/4}\sqrt{1-\frac{k^2-1}{R^2}}+1,&1\le k<\sqrt{1+R^2},\\
    &\phantom{{}<{}}i_k=0,&k\ge\sqrt{1+R^2}.
  \end{align*}
\end{proposition}
\begin{proof}
    We would like to rewrite the bounds in Proposition \ref{prop:fineboundk} in terms of $l$, $r_{\min}$, and $R$. We begin by bounding the expression $\sigma^{-1}\sqrt{1+\frac{1-k^2}{r^2}}$ in terms of $r_{\min}$, $r_{\max}$, $\sigma_{\min}$, and $\sigma_{\max}$, where these denote the minimum and maximum values of $r$ and $\sigma$ on $\Gamma$, respectively. Taking cases based on the sign of $1-k^2$, we obtain whenever $\sqrt{1+\frac{1-k^2}{r^2}}$ is defined that
    \begin{align*}
      \sigma_{\max}^{-1}\sqrt{1+\frac1{r_{\max}^2}}&\le\sigma^{-1}\sqrt{1+\frac{1-k^2}{r^2}}\le\sigma_{\min}^{-1}\sqrt{1+\frac1{r_{\min}^2}},&&k=0,\\
      \sigma_{\max}^{-1}&\le\sigma^{-1}\sqrt{1+\frac{1-k^2}{r^2}}\le\sigma_{\min}^{-1},&&k=1,\\
      &\phantom{{}\le{}}\sigma^{-1}\sqrt{1+\frac{1-k^2}{r^2}}\le\sigma_{\min}^{-1}\sqrt{1-\frac{k^2-1}{r_{\max}^2}},&&1\le k<\sqrt{1+r_{\max}^2}.
    \end{align*}
    A reader might expect the last line to have a lower bound of $\sigma_{\max}^{-1}\sqrt{1-\frac{k^2-1}{r_{\min}^2}}$, but we excluded it because we generally expect $1-\frac{k^2-1}{r_{\min}^2}$ to be negative for $k\ge2$, making this bound vacuous.

    To rewrite these bounds in terms of $r_{\min}$ and $R$, we note that $r_{\min}\le r\le R$, and we can check that
    \begin{equation*}
      \tfrac12r_{\min}e^{-R^2/4}\le\sigma\le\tfrac1{\sqrt{2e}}.
    \end{equation*}
    
    Putting everything together, we have for $k=0$ and $k=1$ that $k^2\le1+r_{\min}^2$, so by Proposition \ref{prop:fineboundk} and using the fact that $2\lfloor a\rfloor+1>2a-1$, we have
    \begin{equation*}
      i_k>\frac l{\pi}\sigma_{\max}^{-1}\sqrt{1+\frac{1-k^2}{r_{\max}^2}}-1\ge\frac l\pi\sqrt{2e}\sqrt{1+\frac{1-k^2}{R^2}}-1.
    \end{equation*}
    One can check that that $i_k>\frac l\pi\sqrt{2e}\sqrt{1+\frac{1-k^2}{R^2}}-1$ continues to hold in the exceptional case of Proposition \ref{prop:fineboundk} because the inequality $\sigma^{-1}\ge\sqrt{2e}$ must be strict on almost all points of $\Gamma$.
    
    As for the upper bound for $k=0$ and $k=1$, using the fact that $2\lceil a\rceil-1<2a+1$, we have
    \begin{equation*}
      i_k<\frac l\pi\sigma_{\min}^{-1}\sqrt{1+\frac{1-k^2}{r_{\min}^2}}+1\le\frac l\pi\frac 2{r_{\min}}e^{R^2/4}\sqrt{1+\frac{1-k^2}{r_{\min}^2}}+1.
    \end{equation*}

    Next, we consider $k\ge1$. If $k<\sqrt{1+r_{\max}^2}$, then by Proposition \ref{prop:fineboundk} and by our bounds on $\sigma^{-1}\sqrt{1+\frac{1-k^2}{r^2}}$, we have
    \begin{equation*}
      i_k<\frac l\pi\sigma_{\min}^{-1}\sqrt{1-\frac{k^2-1}{r_{\max}^2}}+1\le\frac l\pi\frac2{r_{\min}}e^{R^2/4}\sqrt{1-\frac{k^2-1}{R^2}}+1.
    \end{equation*}
    If $\sqrt{1+r_{\max}^2}\le k<\sqrt{1+R^2}$, then $i_k=0$ by Proposition \ref{prop:fineboundk}, so we trivially have $i_k<\frac l\pi\frac2{r_{\min}}e^{R^2/4}\sqrt{1-\frac{k^2-1}{R^2}}+1$. Finally, if $k\ge\sqrt{1+R^2}$, then $k\ge\sqrt{1+r_{\max}^2}$, so $i_k=0$.
\end{proof}

We now prove our coarse estimate for the index.

\begin{proof}[of Theorem \ref{thm:coarsebound}]
  Using Proposition \ref{prop:coarseboundk} and recalling that $i(\Sigma)=-4+i_0+2\sum_ki_k$ to exclude translations and dilations, we compute
  \begin{align*}
    i(\Sigma)&<\frac2{\pi}\frac{F(\Sigma)}{r_{\min}}e^{R^2/4}\left(\sqrt{1+\frac1{r_{\min}^2}}+2\sum_{1\le k<\sqrt{1+R^2}}\sqrt{1-\frac{k^2-1}{R^2}}\right)\\
             &\hphantom{{}<\frac2{\pi}\frac{F(\Sigma)}{r_{\min}}e^{R^2/4}}\qquad{}+1+2\left\lfloor\sqrt{1+R^2}\right\rfloor-4\\
             &<\frac2\pi\frac{F(\Sigma)}{r_{\min}}e^{R^2/4}\left(\sqrt{1+\frac1{r_{\min}^2}}+2\sqrt{1+R^2}\right)+2\sqrt{1+R^2}-3,\\
             &<\frac2\pi\frac{F(\Sigma)}{r_{\min}}e^{R^2/4}\left(3+\frac1{r_{\min}}+2R\right)+2R-1,\\
    i(\Sigma)&>\frac{\sqrt{2e}}\pi F(\Sigma)\left(\sqrt{1+\frac1{R^2}}+2\right)-1-2-4>\frac{3\sqrt{2e}}\pi F(\Sigma)-7.\qedhere
  \end{align*}    
\end{proof}

\subsection{Index bounds for the Angenent torus}\label{subsec:indexbounds} We illustrate the index bounds given in Proposition \ref{prop:fineboundk} by applying them to the Angenent torus computed in \cite{bk19}. We then compare these bounds with our preliminary numerical results for the index of the Angenent torus \cite{bk20}.

In \cite{bk19}, we numerically computed the entropy of the Angenent torus along with a $2048$-point discrete curve that approximates the torus cross-section $\Gamma$. These results have since then been replicated in independent work \cite{bdn19}. We plug this entropy value and the $2048$ values for $(r,z)$ into the formulas in Proposition \ref{prop:fineboundk}, obtaining the following bounds for $i_k$.
\begin{align*}
  1\le i_0&\le7,\\
  1\le i_1&\le5,\\
  i_2&\le5,\\
  i_3&\le3,\\
  i_k&=0,\quad k\ge4.
\end{align*}

We see that our lower bounds are quite weak. Translation and dilation alone gives $i_0\ge2$ and $i_1\ge1$, and Liu's results \cite{l16} improve that to $i_0\ge3$ and $i_1\ge2$. In general, we expect our lower bounds to be useful only for immersed tori with larger entropy. As discussed in the introduction, $F(\Sigma)\ge4.5$ is enough for our coarse lower bounds to give new information, and these finer bounds might give new information even earlier. For comparison, the entropy of the Angenent torus is roughly $1.85$.

Meanwhile, our upper bounds imply that the index of the Angenent torus is at most $7+2(5+5+3)-4=29$, excluding translation and dilation. Our preliminary numerical results \cite{bk20}, however, suggest that the index is actually much closer to Liu's lower bound of $3$. Specifically, we obtained numerically that
\begin{align*}
  i_0&=3,\\
  i_1&=2,\\
  i_2&=1,\\
  i_k&=0,\quad k\ge3,
\end{align*}
for a total index of $3+2(2+1)-4=5$, excluding translation and dilation.

\section{Entropy bounds}\label{sec:entropybounds}
We can use our upper bounds on the index against Liu's lower bound \cite{l16} to obtain information about the entropies of immersed rotationally symmetric self-shrinking tori. From Theorem \ref{thm:coarsebound} and Liu's result that $i(\Sigma)\ge3$, we immediately obtain
\begin{equation*}
  F(\Sigma)>\frac\pi2r_{\min}e^{-R^2/4}\frac{4-2R}{3+r_{\min}^{-1}+2R}.
\end{equation*}
Unfortunately, this is a vacuous bound because we expect the right-hand side to be negative. However, with a more refined version of this approach, we can get nontrivial bounds.

Liu's work and our work both give more information about the eigenvalues of $L_0$ and $L_1$, and with this more detailed analysis, we can get bounds on the entropy. More generally, this section can be seen as illustrating how to exploit information about the eigenvalues of $L_k$ in order to get information about the entropy of $\Sigma$.

The formula for $L_1$ is simpler than that of $L_0$, so we begin here. Horizontal translation gives an eigenfunction of $L_1$ with eigenvalue $-\frac12$, and Liu \cite{l16} shows that there is another eigenfunction of $L_1$ with smaller eigenvalue.
\begin{theorem}\label{thm:fbound1}
  Let $\Sigma\subset\R^3$ be an immersed rotationally symmetric self-shrinking torus. Let $F(\Sigma)$ be the entropy of $\Sigma$, and let $r$ be the distance to the axis of rotation. Then
  \begin{equation*}
    F(\Sigma)\ge\pi\sqrt2\min_\Sigma re^{-\abs x^2/4}.
  \end{equation*}
  Moreover, if $re^{-\abs x^2/4}$ is not constant on $\Sigma$, then the above inequality is strict.
\end{theorem}
\begin{proof}
    As before, let $\Gamma$ be the cross-section of $\Sigma$, and let $l=F(\Sigma)$. By Theorem \ref{thm:lk}, $L_1=\sigma\Delta^\sigma_\Gamma\sigma+1$. We then proceed similarly to the proof of Proposition \ref{prop:maxeigs}.

    Let $u\colon\Gamma\to\R$ be such that $\sigma u$ is orthogonal to $1$ with respect to the $L^2(\Gamma,\sigma)$ pairing. In other words, $0=\int_\Gamma(1)(\sigma u)\,\sigma=\int_\Gamma u\sigma^2$. Then, by Proposition \ref{prop:deltaeigs}, the least eigenvalue of $\Delta_\Gamma^\sigma$ on this orthogonal complement is $\left(\frac{2\pi}l\right)^2$. We then compute that, for such a $u$,
    \begin{equation*}
      \int_\Gamma u(-L_1)u\,\sigma=\int_\Gamma\left(u\sigma(-\Delta_\Gamma^\sigma)\sigma u-u^2\right)\,\sigma\ge\int_\Gamma\left(\left(\frac {2\pi}l\right)^2\sigma^2-1\right)u^2\,\sigma.
    \end{equation*}

    Assume for the sake of contradiction that $l\le\pi\sqrt2 re^{-\abs x^2/4}=\sqrt2(2\pi)\sigma$ at all points of $\Gamma$ and that $l<\sqrt2(2\pi)\sigma$ somewhere on $\Gamma$. Then $\left(\frac{2\pi}l\right)^2\sigma^2\ge\frac12$ on $\Gamma$, and this inequality is strict somewhere on $\Gamma$. We conclude that
    \begin{equation*}
      \int_\Gamma u(-L_1)u\,\sigma>-\frac12\int_\Gamma u^2\,\sigma.
    \end{equation*}
    Thus, on the space of $u$ such that $\int_\Gamma u\sigma^2=0$, all eigenvalues of $L_1$ are larger than $-\frac12$. Because this space of $u$ has codimension one, we see that $L_1$ has at most one eigenvalue less than or equal to $-\frac12$, contradicting Liu \cite{l16}.

    We conclude that $l>\pi\sqrt2 re^{-\abs x^2/4}$ at some point on $\Gamma$ or $l\ge\pi\sqrt2 re^{-\abs x^2/4}$ at all points of $\Gamma$. Either way, we have $l\ge\min_\Gamma\pi\sqrt2re^{-\abs x^2/4}$. If we additionally assume that $re^{-\abs x^2/4}$ is not constant on $\Gamma$, then, either way, we have $l>\min_\Gamma\pi\sqrt2re^{-\abs x^2/4}$.
\end{proof}

Meanwhile, dilation gives an eigenfunction of $L_0$ with eigenvalue $-1$, and Liu \cite{l16} shows that there is another eigenfunction of $L_0$ with smaller eigenvalue. We obtain the corresponding result.
\begin{theorem}\label{thm:fbound0}
  Let $\Sigma\subset\R^3$ be an immersed rotationally symmetric self-shrinking torus. Let $F(\Sigma)$ be the entropy of $\Sigma$, and let $r$ be the distance from the axis of rotation. Then
  \begin{equation*}
    F(\Sigma)\ge\pi\min_\Sigma r^2e^{-\abs x^2/4}.
  \end{equation*}
  Moreover, if $r^2e^{-\abs x^2/4}$ is not constant on $\Sigma$, then the above inequality is strict.
\end{theorem}
\begin{proof}
    We proceed as in the above proof of Theorem \ref{thm:fbound1}, letting $u\colon\Gamma\to\R$ satisfy $\int_\Gamma u\sigma^2=0$. This time, however, we work with the operator $L_0=\sigma\Delta_\Gamma^\sigma\sigma+1+\frac1{r^2}$. We compute that, for such a $u$,
    \begin{multline*}
      \int_\Gamma u(-L_0)u\,\sigma=\int_\Gamma\left(u\sigma(-\Delta_\Gamma^\sigma)\sigma u-\left(1+\frac1{r^2}\right)u^2\right)\,\sigma\\
      \ge\int_\Gamma\left(\left(\frac{2\pi}l\right)^2\sigma^2-1-\frac1{r^2}\right)u^2\,\sigma.
    \end{multline*}

    We assume for the sake of contradiction that $l\le\pi r^2e^{-\abs x^2/4}=2\pi r\sigma$ at all points of $\Gamma$ and that $l<2\pi r\sigma$ somewhere on $\Gamma$. Then $\left(\frac{2\pi}l\right)^2\sigma^2\ge\frac1{r^2}$ on $\Gamma$, and this inequality is strict somewhere on $\Gamma$. We conclude that
    \begin{equation*}
      \int_\Gamma u(-L_0)u\,\sigma>-\int_\Gamma u^2\,\sigma.
    \end{equation*}
    Thus, on the space of $u$ such that $\int_\Gamma u\sigma^2=0$, all eigenvalues of $L_0$ are larger than $-1$. We conclude that there is at most one eigenvalue of $L_0$ that is less than or equal to $-1$, contradicting Liu's result \cite{l16} that, in addition to the eigenfunction $H$ with eigenvalue $-1$, there is another eigenfunction of $L_0$ with eigenvalue strictly smaller than $-1$.

    From here, the result follows in the same way as in the proof of Theorem \ref{thm:fbound1}.
\end{proof}

\section{Two explicit eigenfunctions of the stability operator}\label{sec:eigenfunctions}
In this section, we present an easy but surprising consequence of our formula for $L_k$ in Theorem \ref{thm:lk}. Recall that Liu showed that $L_1$ has an eigenfunction with eigenvalue strictly smaller than $-\frac12$, thereby giving a pair of entropy-decreasing variations of rotationally symmetric self-shrinking tori \cite{l16}. We find that this eigenfunction is simply $\sigma^{-1}$, and its eigenvalue is $-1$.

\begin{theorem}\label{thm:l1eig}
  Let $\Sigma\subset\R^3$ be an immersed rotationally symmetric self-shrinking torus. Then
  \begin{equation*}
    \frac2re^{\abs x^2/4}\cos\theta\qquad\text{and}\qquad\frac2re^{\abs x^2/4}\sin\theta
  \end{equation*}
  are eigenfunctions of the stability operator $L_\Sigma$ with eigenvalue $-1$.
\end{theorem}
\begin{proof}
    Note that $\frac2re^{\abs x^2/4}=\sigma^{-1}$. By Theorem \ref{thm:lk},
    \begin{equation*}
      L_1\sigma^{-1}=\sigma\Delta_\Gamma^\sigma(\sigma\sigma^{-1})+\sigma^{-1}=\sigma\Delta_\Gamma^\sigma1+\sigma^{-1}=\sigma^{-1}.
    \end{equation*}
    Thus, by Proposition \ref{prop:Lfourierdecomposition}
    \begin{equation*}
      L_\Sigma\left(\sigma^{-1}\cos\theta\right)=L_1\left(\sigma^{-1}\right)\cos\theta=\sigma^{-1}\cos\theta,
    \end{equation*}
    and likewise for $\sigma^{-1}\sin\theta$.
\end{proof}

We present plots of this variation in Figures \ref{fig:variationQuiver} and \ref{fig:variation3d}.

\begin{figure}
  \centering
  \includegraphics[scale=.8]{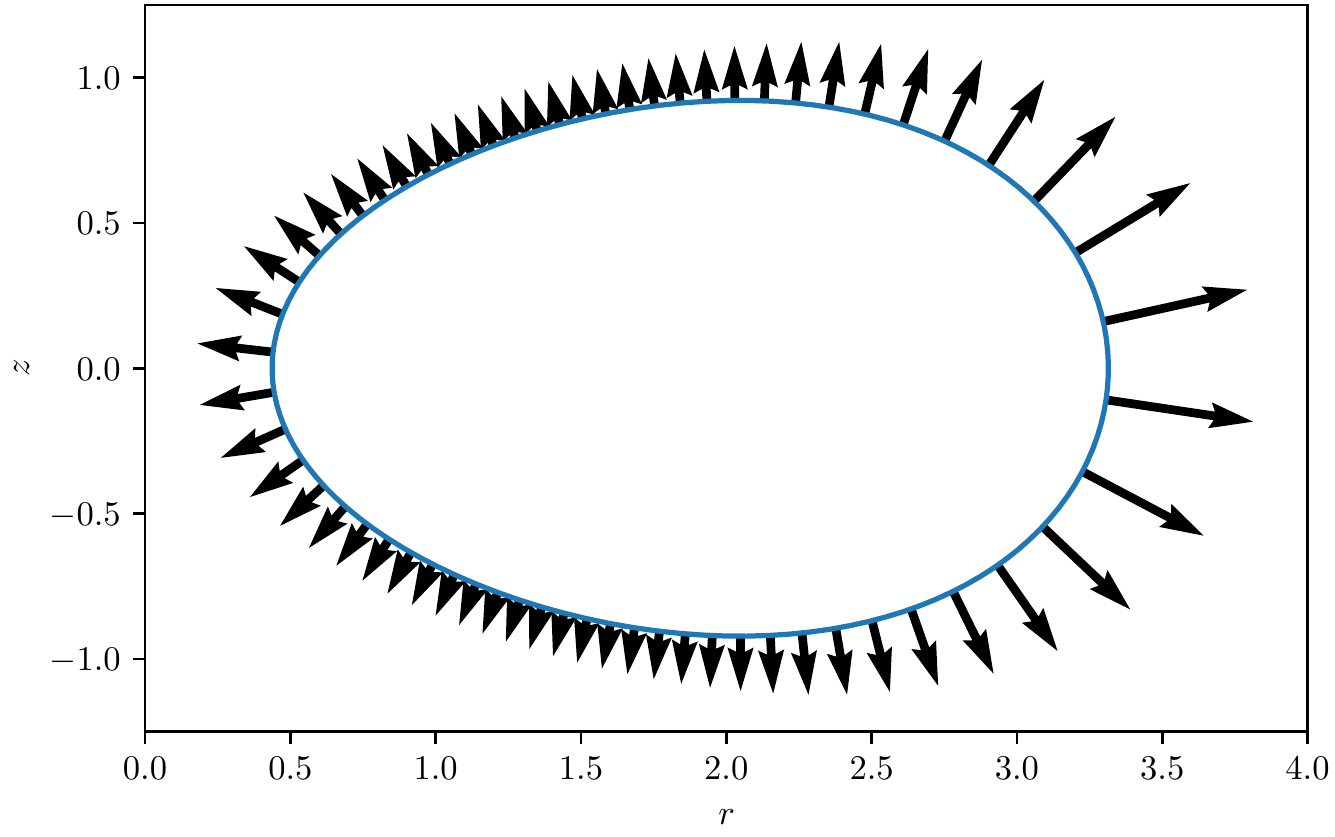}
  \caption{The normal variation of the Angenent torus cross-section given by $u=\sigma^{-1}=\frac2re^{\abs x^2/4}$.}
  \label{fig:variationQuiver}
\end{figure}

\begin{figure}
  \centering
  \includegraphics[trim = 30mm 50mm 30mm 50mm, clip]{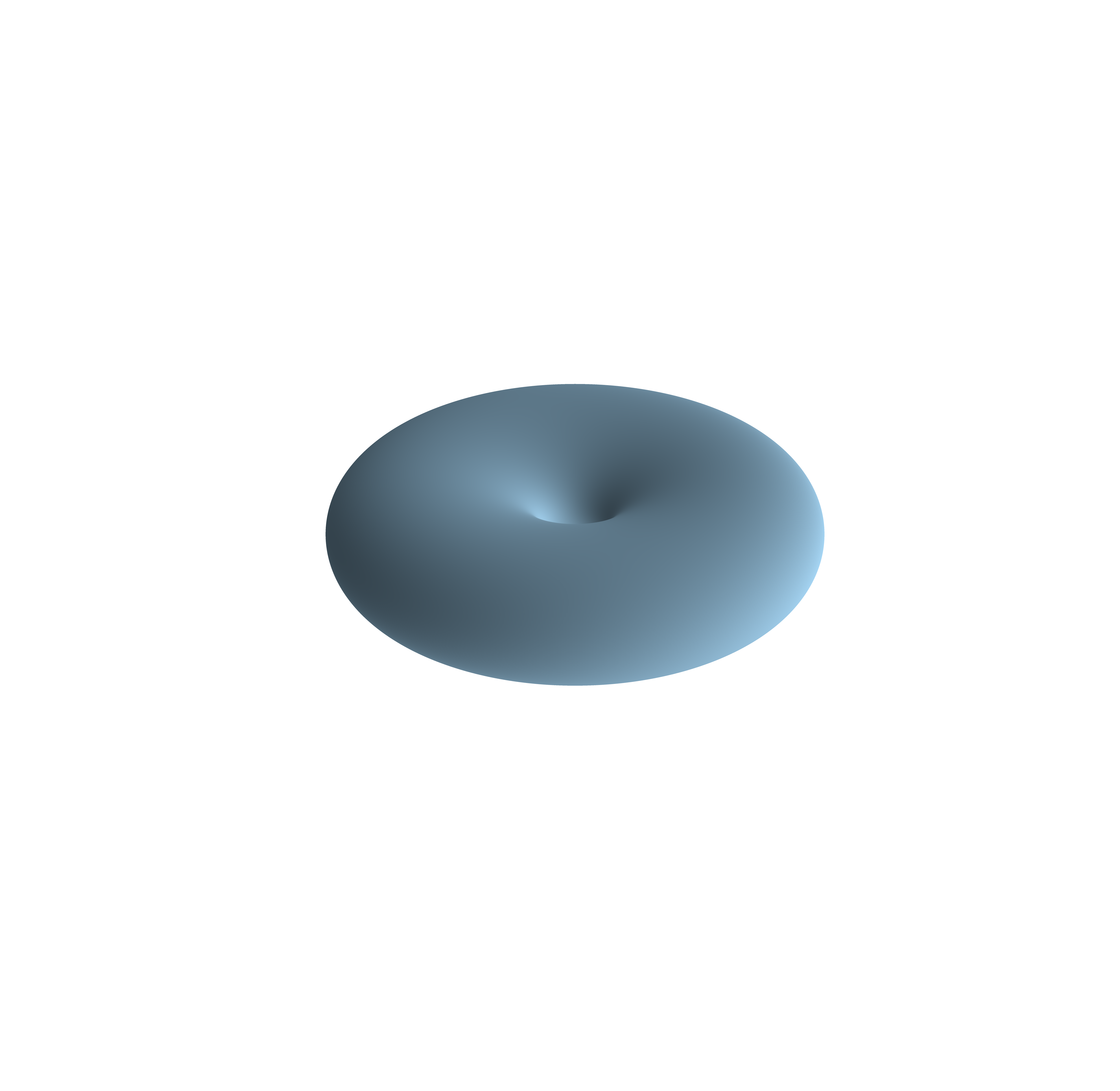}
  \includegraphics[trim = 30mm 50mm 30mm 50mm, clip]{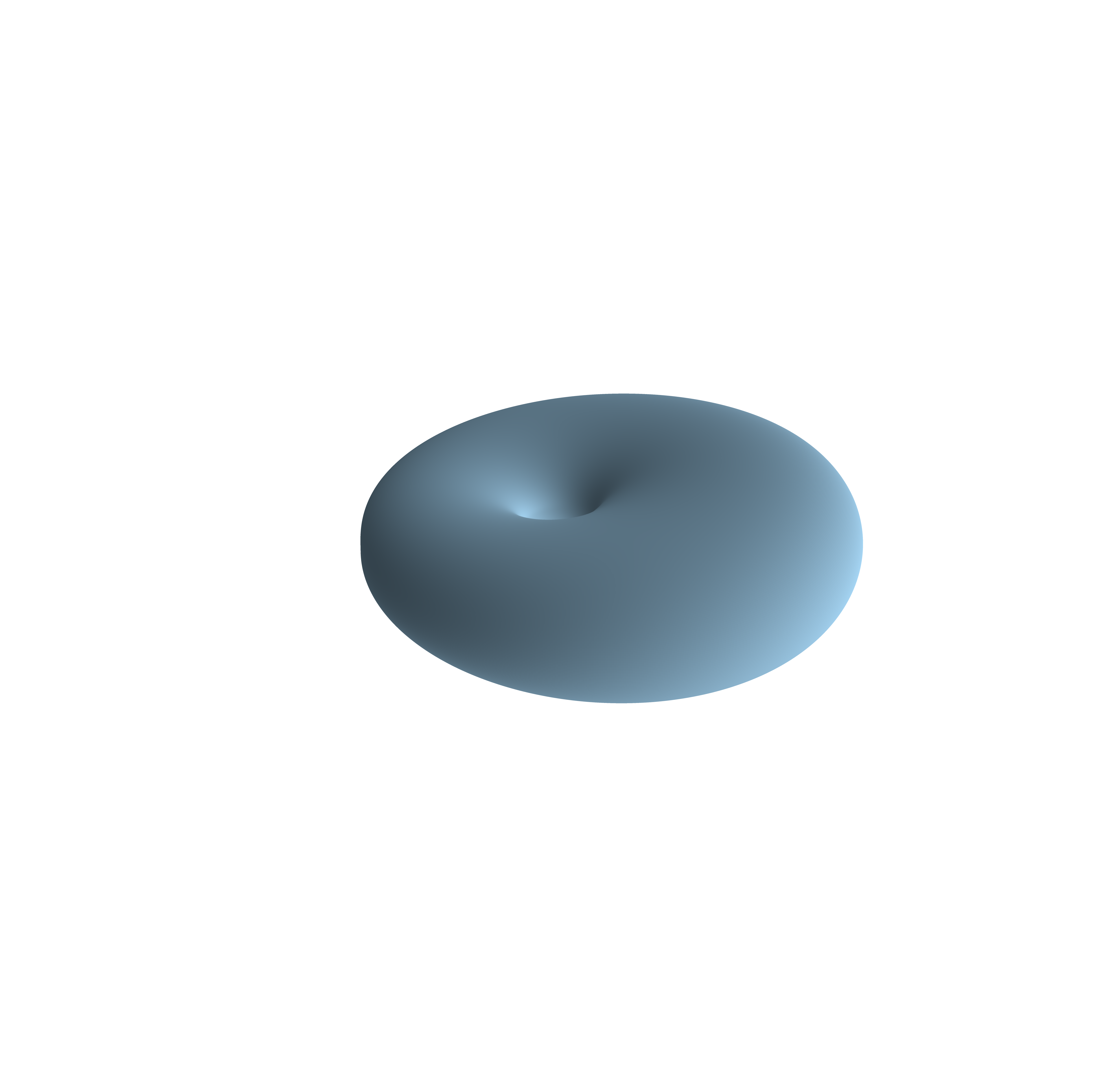}
  \caption{The Angenent torus (top) and the Angenent torus deformed by the normal variation $f=\sigma^{-1}\cos\theta=\frac2re^{\abs x^2/4}\cos\theta$ (bottom).}
  \label{fig:variation3d}
\end{figure}

\section{Preliminary results in higher dimensions}\label{sec:highdim}
In higher dimensions, the situation is more complicated. One of the new features is that the scalar curvature of the cross-section $\Gamma$ plays an essential role. We present some preliminary results that illustrate its impact.

Our main result in this section is Proposition \ref{prop:L1u}, where we compute $L_1\sigma^{-1}$ in general dimension, generalizing Theorem \ref{thm:l1eig}. The ``preliminary'' aspect of this result is that we prove it by direct computation, rather than by generalizing Theorem \ref{thm:lk} using the second variation formula for minimal hypersurfaces with respect to an appropriate conformally changed metric. We also present some consequences of Proposition \ref{prop:L1u} regarding the spectrum of the stability operator $L_\Sigma$.

\begin{proposition}\label{prop:L1u}
  Let $\Gamma$ be the cross-section of an immersed compact $SO(2)$-rotationally symmetric self-shrinking hypersurface $\Sigma\subset\R^{n+1}$ that does not intersect the axis of rotation $r=0$. Let $u\colon\Gamma\to\R$ be defined by
  \begin{equation*}
    u=\sigma^{-1}=\frac Cre^{\abs x^2/4},
  \end{equation*}
  where $C=\frac{(4\pi)^{n/2}}{2\pi}$. Then
  \begin{equation*}
    L_1u=\left(\frac n2-R_\Gamma\right)u,
  \end{equation*}
  where $R_\Gamma$ is the scalar curvature of $\Gamma$.
\end{proposition}
\begin{proof}
    Consider a point $p\in\Gamma$ and choose Riemannian normal coordinates about $p$, giving us a coordinate frame $e_1,\dotsc,e_{n-1}$ for $\Gamma$ that is orthonormal at $p$ and satisfies $\nabla_i{e_j}=a_{ij}\n$ at $p$. We compute at an arbitrary point in the coordinate chart that
    \begin{equation*}
      \nabla_iu=C\left(-\frac1{r^2}\langle e_r,e_i\rangle e^{\abs x^2/4}+\frac1re^{\abs x^2/4}\langle \tfrac12x,e_i\rangle\right)=\left\langle\frac12x-\frac1re_r,e_i\right\rangle u.
    \end{equation*}
    Then, at $p$, we compute that
    \begin{equation*}
      \begin{split}
        \L_\Gamma u&=\sigma^{-1}\div_\Gamma(\sigma\grad_\Gamma u)\\
        &=\sigma^{-1}\sum_{i=1}^{n-1}\nabla_i(\sigma\nabla_iu)\\
        &=\sigma^{-1}\sum_{i=1}^{n-1}\nabla_i\left(\sigma\left\langle\frac12x-\frac1re_r,e_i\right\rangle u\right)\\
        &=\sigma^{-1}\sum_{i=1}^{n-1}\nabla_i\left\langle\frac12x-\frac1re_r,e_i\right\rangle\\
        &=u\sum_{i=1}^{n-1}\left(\frac12\langle\nabla_ix,e_i\rangle+\frac1{r^2}\langle e_r,e_i\rangle^2-\frac1r\langle\nabla_ie_r,e_i\rangle+\left\langle\frac12x-\frac1re_r,\nabla_ie_i\right\rangle\right)\\
        &=u\sum_{i=1}^{n-1}\left(\frac12\abs{e_i}^2+\frac1{r^2}\langle e_r,e_i\rangle^2-0+a_{ii}\left\langle\frac12x-\frac1re_r,\n\right\rangle\right)\\
        &=\left(\frac{n-1}2+\frac1{r^2}\lvert e_r^\top\rvert^2-H_\Gamma\left(\frac12x^\perp-\frac1re_r^\perp\right)\right)u.
      \end{split}
    \end{equation*}
    Note that our final expression is coordinate-independent and is thus valid at any point $p$ in $\Gamma$.

    If $\Gamma$ is the cross-section of a self-shrinker, then $H_\Gamma=\frac12x^\perp-\frac1re_r^\perp$. Consequently, we have that
    \begin{equation*}
      \L_\Gamma u=\left(\frac{n-1}2+\frac1{r^2}\lvert e_r^\top\rvert^2-H_\Gamma^2\right)u,
    \end{equation*}
    Therefore, by Definition \ref{def:lk},
    \begin{equation*}
      \begin{split}
        L_1u&=\left(\frac{n-1}2+\frac1{r^2}\lvert e_r^\top\rvert^2-H_\Gamma^2+\abs{A_\Gamma}^2+\frac1{r^2}\lvert e_r^\perp\rvert^2+\frac12-\frac1{r^2}\right)u\\
        &=\left(\frac n2+\frac1{r^2}\left(\lvert e_r^\top\rvert^2+\lvert e_r^\perp\rvert^2-1\right)-\left(H_\Gamma^2-\abs{A_\Gamma}^2\right)\right)u,\\
        &=\left(\frac n2-R_\Gamma\right)u,
      \end{split}
    \end{equation*}
    as desired. The formula $R=H^2-\abs A^2$ follows from the Gauss--Codazzi equations.
\end{proof}

In particular, if $n=2$, then $\Gamma$ is one-dimensional, so $R_\Gamma=0$, and thus we recover the result $L_1u=u$ from Theorem \ref{thm:l1eig}.

Using Proposition \ref{prop:L1u}, we can give bounds on the least eigenvalue of the stability operator.

\begin{proposition}
  Let $\Sigma\subset\R^{n+1}$ be an immersed compact rotationally symmetric self-shrinking hypersurface that does not intersect the axis of rotation $r=0$. Then the least eigenvalue $\lambda_{1,0}$ of the first Fourier component $L_1$ of the stability operator satisfies the inequality
  \begin{equation*}
    \lambda_{1,0}\le-\frac n2+\max_\Gamma R_\Gamma,
  \end{equation*}
  where $R_\Gamma$ denotes the scalar curvature of the cross-section $\Gamma$ of $\Sigma$.
\end{proposition}
\begin{proof}
    Let $u=\sigma^{-1}$. Then by Proposition \ref{prop:L1u},
    \begin{equation*}
      \int_\Gamma u(-L_1)u\,\sigma=\int_\Gamma u\left(-\frac n2+R_\Gamma\right)u\,\sigma
      \le\left(-\frac n2+\max R_\Gamma\right)\int_\Gamma u^2\,\sigma.
    \end{equation*}
    Thus,
    \begin{equation*}
      \lambda_{1,0}=\inf_{v\neq0}\frac{\int_\Gamma v(-L_1)v\,\sigma}{\int_\Gamma v^2\,\sigma}\le-\frac n2+\max R_\Gamma,
    \end{equation*}
    as desired.
\end{proof}

\begin{proposition}
  Let $\Sigma\subset\R^{n+1}$ be an immersed compact rotationally symmetric self-shrinking hypersurface that does not intersect the axis of rotation $r=0$. Then the least eigenvalue $\lambda_0$ of the stability operator $L_\Sigma$ satisfies the inequality
  \begin{equation*}
    \lambda_0\le-\frac n2+\max_\Gamma\left(R_\Gamma-\frac1{r^2}\right),
  \end{equation*}
  where $R_\Gamma$ denotes the scalar curvature of the cross-section $\Gamma$ of $\Sigma$, and $r$ denotes the distance to the axis of rotation.
\end{proposition}
\begin{proof}
    The eigenfunction corresponding to the lowest eigenvalue of $L$ cannot change sign and must therefore be rotationally symmetric. Thus, it is the least eigenfunction of $L_0$, so
    \begin{equation*}
      \lambda_0=\inf_{v\neq0}\frac{\int_\Gamma v(-L_0)v\,\sigma}{\int_\Gamma v^2\,\sigma}.
    \end{equation*}
    Letting $u=\sigma^{-1}$, we compute that
    \begin{multline*}
      \int_\Gamma u(-L_0)u\,\sigma=\int_\Gamma u\left(-L_1-\frac1{r^2}\right)u\,\sigma=\int_\Gamma u\left(-\frac n2+R_\Gamma-\frac1{r^2}\right)u\,\sigma\\
      \le\left(-\frac n2+\max_\Gamma\left(R_\Gamma-\frac1{r^2}\right)\right)\int_\Gamma u^2\,\sigma.\qedhere
    \end{multline*}
\end{proof}

\section{Future work}\label{sec:future}
We end by presenting six possible future directions for this work.

\subsection{General rotationally symmetric surfaces} In this paper, we restricted our attention to tori in order to avoid the places where the metric $g^\sigma$ becomes degenerate, namely the axis of rotation and infinity. However, in principle, we could do a similar analysis for rotationally symmetric immersed spheres, cylinders, and planes. In these contexts, the formula for $L_k$ given in Theorem \ref{thm:lk} remains valid, and the cross-section $\Gamma$ still has finite length with respect to $g^\sigma$. However, rather than being a circle, the cross-section $\Gamma$ is now an interval with boundary either on the axis of rotation or at infinity. Thus, since $\sigma=0$ on the boundary of $\Gamma$, we would now consider eigenvalues of the Laplacian with Dirichlet boundary conditions rather than periodic boundary conditions.

However, the analysis looks like it will be much more delicate than in this paper. For example, if $\Sigma$ is the round sphere, then dilation is given by $f=1$, so $L_0(1)=1$. However, our formula gives $L_0(1)=\sigma\Delta_\Gamma^\sigma\sigma+1+\frac1{r^2}$, with two unbounded terms that cancel.

\subsection{Higher-dimensional hypersurfaces with $SO(2)$ symmetry}\label{subsec:highdim} Liu's work \cite{l16} applies to self-shrinkers with $SO(2)$ rotational symmetry in any dimension. The fact that our $L_1\sigma^{-1}=\sigma^{-1}$ result in Section \ref{sec:eigenfunctions} generalized nicely to higher dimensions in Proposition \ref{prop:L1u} suggests that our index bound results in Section \ref{sec:indexbounds} might generalize as well. Rather than being a geodesic with respect to a conformally changed metric, the cross-section $\Gamma$ would now be a minimal surface, and we can use the second variation formula for minimal hypersurfaces in a Riemannian manifold.

As in Proposition \ref{prop:L1u}, we expect the scalar curvature of the cross-section $\Gamma$ to play an important role in higher dimensions. Another complication is that, in general dimension, the area of a manifold does not determine the spectrum of its Laplacian. We would need to use bounds on the eigenvalues of the Laplacian instead.

\subsection{General self-shrinkers} In this paper, we viewed the cross-section $\Gamma$ as a geodesic, or, more generally, a minimal hypersurface, with respect to a conformally changed metric. However, there is no need to have rotational symmetry to take this perspective. As noted in \cite{cm12}, the self-shrinker $\Sigma$ is itself a minimal surface with respect to an appropriate conformally changed metric $g^\sigma$. Like before, the entropy of $\Sigma$ is simply its area with respect to this metric. (Note that this $\sigma$ is not the same conformal factor as the one in this paper.)

We could then apply the key idea of our paper: Use the second variation formula for minimal surfaces in a Riemannian manifold in order to obtain a new formula for the stability operator $L_\Sigma$ in terms of the intrinsic Laplacian $\Delta_\Sigma^\sigma$. As in the future work described in \ref{subsec:highdim}, we expect new features and challenges to appear, even for surfaces: The scalar curvature $R_\Sigma$ will play a part, and the eigenvalues of $\Delta_\Sigma^\sigma$ will have bounds in terms of the entropy rather than explicit formulas.

We also expect some things to be easier in the general setting. Specifically, we will not need to consider the minimum distance $r_{\min}$ to the axis of rotation, and, indeed, in the general context, the metric $g^\sigma$ is only degenerate at infinity. We expect the role of $r_{\min}$ to instead be played by a lower bound on the scalar curvature $R_\Sigma$, based on the fact that, for a rotationally symmetric self-shrinking torus, the points where $r$ is the smallest are also the points where $R_\Sigma$ is the most negative.

\subsection{Self-shrinkers with more symmetry} Liu's work \cite{l16} is about self-shrinkers with $SO(2)$ rotational symmetry. However, Angenent's work \cite{a92} is about self-shrinking doughnuts in $\R^{n+1}$ with $SO(n)$ rotational symmetry, and recently there has been work on self-shrinking hypersurfaces in $\R^{n+m}$ with birotational $SO(n)\times SO(m)$ symmetry \cite{dln18,m15}. Certainly, Liu's work can be generalized to $SO(2)\times SO(2)$ birotational symmetry, and thus our work can be generalized in this way as well. For the more general problem, one would have to replace Liu's decomposition of the stability operator into its Fourier components with a decomposition into $SO(n)$-representations.

\subsection{Flow lines corresponding to the new eigenfunctions} The new eigenfunctions with eigenvalue $-1$ that we found in Section \ref{sec:eigenfunctions} raise an interesting question. All of the other known eigenfunctions with rational negative eigenvalues are tangent to geometrically meaningful unstable flow lines for rescaled mean curvature flow, namely translations and dilations. Are the flow lines corresponding to the eigenfunctions in Theorem \ref{thm:l1eig} geometrically meaningful as well? Figure \ref{fig:variation3d} suggests that there may be some Möbius-like transformation happening, but it is far from clear whether these flow lines have explicit formulas, and, if so, what they are.

The conjecture \cite{km14} that the Angenent torus is the only \emph{embedded} rotationally symmetric self-shrinking torus is still open. Any such torus, whether embedded or only immersed, has these flow lines, so perhaps understanding them better could help constrain rotationally symmetric self-shrinking tori enough to resolve this conjecture. For recent progress towards this conjecture via compactness theorems, see \cite{m20}.

\subsection{Computing the index numerically} In \cite{bk19}, we numerically computed the entropy of the Angenent torus and proposed a research program for numerically computing its index and for proving bounds on the accuracy of the entropy and index computations. In a recent preprint \cite{bk20}, we computed the index of the Angenent torus to be $5$, excluding translation and dilation. The index upper bound in Section \ref{sec:indexbounds} is an important ingredient for proving the accuracy of this value.

Specifically, this bound shows that we can restrict our search for entropy-decreasing variations to a finite-dimensional space. Moreover, we describe this finite-dimensional space of variations explicitly in the proof of Proposition \ref{prop:maxeigs}, in terms of Fourier modes. In doing so, we have reduced the problem of finding the number of negative eigenvalues of an infinite-dimensional operator to the problem of finding the number of negative eigenvalues of a finite-dimensional matrix. Based on Subsection \ref{subsec:indexbounds}, for the Angenent torus, we need only to compute the number of negative eigenvalues of a $7\times7$ matrix, two $5\times5$ matrices, and one $3\times3$ matrix. Doing so is a trivial task for a computer, so it remains only to show that we computed the entries of these matrices to sufficiently high accuracy.

\bibliographystyle{plain}
\bibliography{meancurvature}

\end{document}